\documentclass[11pt]{amsart} 
\usepackage{amsmath, amssymb, latexsym, tikz}
\usepackage{caption,ifthen,url,fancyhdr,cite}
\usepackage[foot]{amsaddr}    
\usepackage{libertine}     
\usepackage{longtable} 
\usepackage{upquote}
  
\usepackage{listings}  
\usepackage{textcomp}
\usepackage{adjustbox}

\usepackage[OT2,T1]{fontenc}
\DeclareSymbolFont{cyrletters}{OT2}{wncyr}{m}{n}
\DeclareMathSymbol{\Sha}{\mathalpha}{cyrletters}{"58}

\makeatletter 
\renewcommand\subsubsection{\@startsection{subsubsection}{3}%
  \z@{.5\linespacing\@plus.7\linespacing}{-.5em}%
  {\normalfont\bfseries}} 
\makeatother
 
\usepackage[T1]{fontenc}

\newcommand\makequotestraight{%
\begingroup\lccode`~=`' 
\lowercase{\endgroup\let~}\textquotesingle
\catcode`'=\active
}

\usepackage[a4paper,right=3cm, left=3cm, top=3cm, bottom=3cm, marginpar=1.6cm]{geometry}

\newtheoremstyle{mythm}                   
{6pt}
{6pt}
{\it}
{}
{\bf}
{.}
{.5em}
{}

\newtheoremstyle{mydef}                   
{6pt}
{6pt}
{}
{}
{\bf}
{.}
{.5em}
{}

\newtheoremstyle{myrem}                   
{6pt}
{6pt}
{}
{}
{\bf}
{.}
{.5em}
{}

\theoremstyle{mythm}      
\newtheorem{theorem}{Theorem}[section]
\newtheorem{proposition}[theorem]{Proposition}

\newtheorem{corollary}[theorem]{Corollary}

\theoremstyle{mydef}      
\newtheorem{definition}[theorem]{Definition}
\newtheorem{example}[theorem]{Example}

\theoremstyle{myrem}
\newtheorem{remark}[theorem]{Remark}
\numberwithin{equation}{section}

\newcounter{ithmcount}

\newenvironment{ithm}{\begin{list}{{\rm \alph{ithmcount})}}{\usecounter{ithmcount}\labelwidth18pt
      \leftmargin18pt \topsep3pt \itemsep1pt \parsep2pt}}{\end{list}}

\renewcommand{\leq}{\leqslant} 
\renewcommand{\geq}{\geqslant}

\newcommand{\Q}{\mathbb{Q}}
\newcommand{\Z}{\mathbb{Z}}

\newcommand{\ad}{\mathrm{\mathop{ad}}}

\newcommand{\rank}{\mathrm{\mathop{rank}}}
\newcommand{\gl}{\mathfrak{\mathop{gl}}}

\newcommand{\g}{\mathfrak{g}}

\newcommand{\Der}{{\rm Der}}
\newcommand{\Inn}{{\rm Inn}}
\newcommand{\AID}{{\rm AID}}
\newcommand{\CAID}{{\rm CAID}}

\makeatletter
\newcommand{\udot}{\mathpalette\udot@\relax}
\newcommand{\udot@}[2]{%
  \begingroup
  \sbox\z@{$#1{:}$}%
  \sbox\tw@{$#1{.}$}%
  \raisebox{\dimexpr\ht\z@-\ht\tw@}{$\m@th#1.$}%
  \endgroup
}
\makeatother

\usepackage{fancyvrb}
\fvset{commandchars=\\\{\}}

\begin{document}

\title{A computational approach to  almost-inner derivations}
\subjclass[2000]{} 
\author[H. Dietrich]{Heiko Dietrich}
\author[W. A. de Graaf]{Willem A.\ de Graaf}
\address[Dietrich]{School of Mathematics, Monash University, Clayton VIC 3800, Australia}
\address[de Graaf]{Department of Mathematics, University of Trento, Povo (Trento), Italy}
\email{\rm heiko.dietrich@monash.edu, willem.degraaf@unitn.it}
\thanks{The authors thank Boris Kunyavskii for proposing this question to us (at the Monash Prato Workshop `Lie Theory: frontiers, algorithms, and applications'), and suggesting to look at the example of~\cite{Sah}.}
\date{\today}

\begin{abstract}
We present a computational approach to determine the space of almost-inner derivations of a finite dimensional Lie algebra given by a structure constant table. We also present an example of a Lie algebra for which the quotient algebra of the almost-inner derivations modulo the inner derivations is non-abelian. This answers a question of Kunyavskii and  Ostapenko.
\end{abstract}

\maketitle

\section{Introduction}\label{sec_intro} 

\noindent Let $\g$ be a finite dimensional Lie algebra over a field $F$. A linear map $\delta\colon \g\to \g$ is a {derivation} if for all $a,b\in \g$ it satisfies
\[\delta([a,b])=[\delta(a),b]+[a,\delta(b)].\]The set of all derivations of $\g$ forms an $F$-vector space, denoted $\Der(\g)$, and  a short calculation shows that $\Der(\g)$ is a Lie-subalgebra of $\gl(\g)$. For every $a\in\g$ the adjoint homomorphism $\ad(a)\colon \g\to\g$, $b\mapsto [a,b]$, is a derivation; these are the inner derivations of $\g$ and they form a subalgebra $\Inn(\g)$ of $\Der(\g)$; note that $[\ad(a),\ad(b)](x)=\ad([a,b])(x)$. In particular, the map $\ad\colon \g\to\Inn(\g)$ is a surjective homomorphism whose kernel is the centre $\mathfrak{z}(\g)$ of $\g$, and $\Inn(\g)$ is spanned by $\ad(b)$ where $b$ runs over an $F$-basis of $\g$. If $\delta$ is a derivation of $\g$, then $[\delta,\ad(a)]=\ad(\delta(a))$ for all $a\in\g$, which shows that $\Inn(\g)$ is an ideal of $\Der(\g)$. A derivation $\delta$ of $\g$ is {almost-inner} if there exists a map $A_\delta\colon \g\to\g$ such that for every $a\in \g$ it satisfies
\[\delta(a)=[A_\delta(a),a],\]that is, $\delta(a)\in [\g,a]$ for all $a\in \g$. The map $A_\delta$ is neither unique nor linear in general; for example, one can modify each image of $A_\delta$ by adding a different central element of $\g$.  The $F$-space of all almost-inner derivations on $\g$ is denoted $\AID(\g)$. An almost-inner derivation $\delta$ is {central almost-inner} if there is some $a\in \g$ such that $\delta-\ad(a)$ maps $\g$ into the centre of $\g$. We follow the convention in \cite{burde18} and denote the space of central almost inner derivation of $\g$ by $\CAID(\g)$.  Since $\ad(a)(b)=[a,b]$ for every $a,b\in\g$, the inner derivation $\ad(a)$ is an almost-inner derivation with constant map $A_{\ad(a)}=a$. More generally, in \cite[Proposition 2.3]{burde18} the following inclusion of Lie subalgebras of $\Der(\g)$ is shown 
\[\Inn(\g)\leq \CAID(\g)\leq \AID(\g)\leq \Der(\g).\]
Recall that $\Inn(\g)=\Der(\g)$ for every semisimple Lie algebra over a field of characteristic 0, see for example \cite[Theorem 5.3]{hum}. 

Clearly, $\Inn(\g)$ is an ideal in each of these subalgebras. It is shown in \cite[Proposition 2.4]{burde18} that $\CAID(\g)$ is an ideal in $\AID(\g)$, but it remains open whether or not $\AID(\g)$ is an ideal in $\Der(\g)$: this is conjectured to be true in  \cite[Remark 2.5]{burde18}. For more details on known results, we refer to \cite{burde18,burde21}; for example, it is known that $\AID(\g)=\CAID(\g)=\Inn(\g)$ for every complex Lie algebra $\g$ of dimension at most $4$, see \cite[Proposition 2.8]{burde18}. 

Almost-inner derivations have first been considered by Gordon and Wilson \cite{gordon} in a differential-geometric context. They have recently been studied by Saeedi and collaborators (see, e.g.\ \cite{saeedi15,saeedi18}) and Burde, Dekimpe, and Verbeke (see, e.g.\ \cite{burde18, burde21, phd}). Most recently, Kunyavskii and Ostapenko \cite{boris} used $\AID(\g)$ to define an algebra-theoretic analog of the Tate-Shafarevich group, the {Tate-Shafarevich algebra} of a Lie algebra $\g$,
\[\Sha(\g)=\AID(\g)/\Inn(\g),\]
see \cite[Section 2]{boris}. They point out that algebras with nonzero $\Sha(\g)$  reveal important geometric phenomena. One of their main results is the proof that $\AID(\g)$ is an ideal of $\Der(\g)$ for nilpotent~$\g$, partially answering the aforementioned conjecture affirmatively. This  also implies that $\Sha(\g)$ is an ideal of ${\rm Out}(\g)=\Der(\g)/\Inn(\g)$ for nilpotent $\g$, see \cite[Theorem 2.5]{boris}.

The first author of \cite{boris}  asked us for computational methods to determine the subalgebra $\AID(\g)$, and whether there is a Lie algebra $\g$ for which $\AID(\g)/\Inn(\g)$ is non-abelian, see \cite[Question 4.1(i)]{boris}. We develop such an approach in Section \ref{sec_comp}, and then discuss some computational examples (including an affirmative answer to the question) in Sections \ref{sec_ex} and \ref{secsha}.

\section{Computational approach}\label{sec_comp}
\noindent Let $\g$ be a Lie algebra over a  field $F$, with basis $\mathcal{B}=\{b_1,\ldots,b_n\}$ and corresponding structure constants $\sigma_{i,j}^k$, that is, for each $i,j\in\{1,\ldots,n\}$ we have  $[b_i,b_j]=\sum_{k=1}^n \sigma_{i,j}^k b_k$. Let $\delta$ be an endomorphism of $\g$, represented by an $n\times n$ matrix with entries $d_{j,k}$, such that $\delta(b_j)=\sum_{k=1}^n d_{j,k} b_k$ for each $i$. Since $\delta$ is a derivation if and only if $\delta([b_i,b_j])=[\delta(b_i),b_j]+[b_i,\delta(b_j)]$ for all $i,j$, this  translates to the following equations for each $i,j,\ell\in\{1,\ldots,n\}$:
\begin{eqnarray*}\sum\nolimits_{k=1}^n (\sigma_{i,j}^k d_{k,\ell}-\sigma_{k,j}^\ell d_{i,k}-\sigma_{i,k}^\ell d_{j,k})=0.
\end{eqnarray*}Since $[b_i,b_j]=-[b_j,b_i]$, it suffices to consider $i>j$. This shows that a basis for the derivation algebra $\Der(\g)$ can be computed by solving a system of $n^2(n+1)/2$ linear equations with $n^2$ unknowns, see also \cite[Section 1.9]{graaf}. A basis of $\Inn(\g)$ as a subalgebra can readily be computed by determining the matrices (with respect to $\mathcal{B})$ of the adjoints $\ad(b_i)$ for every $i\in\{1,\ldots,n\}$.

In conclusion, linear algebra methods can be used to compute $\Der(\g)$ and $\Inn(\g)$, and therefore also a complement subspace $U\leq \Der(\g)$ such that \[\Der(\g)=\Inn(\g)\oplus U.\] In the sequel we fix one such space $U$. Since $\Inn(\g)\leq \AID(\g)$, to determine $\AID(\g)$, it suffices to compute the space $U\cap \AID(\g)$.
 
\begin{definition}\label{def_D}
  For $z_0\in \g$ let $D_{z_0}$ be the subspace of $U$ that consists of all derivations $\delta$ that act as an inner derivation on the $1$-dimensional subspace spanned by $z_0$; note that $\delta$ is not inner since $U\cap\Inn(\g)=\{0\}$. In other words, $\delta\in D_{z_0}$ if and only if $\delta\in U$ and $\delta(z_0)=[z_0,x]$ for some $x\in \g$.
\end{definition}
If $z_0\in\g$ and  $\psi_{z_0}$ denotes the linear map
  \[\psi_{z_0} \colon U\oplus \g \to \g,\quad (\delta,x)\mapsto  \delta(z_0)-[z_0,x],\]then $D_{z_0}$ is the image of the kernel $\ker \psi_{z_0}$ under the projection $U\oplus \g \to U$.  This shows that we can compute $D_{z_0}$ by solving a linear equation system. In particular,
\[U\cap \AID(\g) = \bigcap\nolimits_{z_0\in \g} D_{z_0}.\]
Clearly, a finite intersection suffices to construct $U\cap \AID(\g)$ in this way. However, it is not clear how to chose suitable elements $z_0$, and how to establish that the intersection is as small as possible.

We fix the previous notation and, throughout, let $V$ be a subspace of $U$, for example, the intersection of multiple spaces $D_{z_0}$ for arbitrarily chosen elements $z_0\in \g$, so that \[\AID(\g)\leq\Inn(\g)\oplus V.\]Let  $\{\delta_1,\ldots,\delta_s\}$ be a basis of $V$. How to decide whether $\delta=\sum_{i=1}^s d_i\delta_i\in V$ lies in $\AID(\g)$? The following proposition provides an answer, but we need some notation before we can state it.

For a vector $z=(z_1,\ldots,z_n)\in F^n$ define the $n\times n$ matrix $M(z)$ as
\[M(z) = (m_{k,j}(z))_{k,j}\quad\text{where each}\quad m_{k,j}(z)=\sum\nolimits_{i=1}^n z_i\sigma_{i,j}^k;\]
recall that the $\sigma_{i,j}^k$ are the structure constants with respect to the basis $\{b_1,\ldots,b_n\}$ of $\g$. Also define $b_z\in\g$ by \[b_z=z_1b_1+\cdots +z_nb_n.\] For $\delta\in V$ write $\delta(b_z)=c_1(z)b_1+\cdots +c_n(z)b_n$ with each $c_i(z)\in F$, and define the column vector $v_\delta(z)$ as
\[v_\delta(z) = (c_1(z),\ldots,c_n(z))^\intercal.\] Lastly, denote by \[M_\delta(z)=[M(z)|v_\delta(z)]\] the augmented matrix $M(z)$ with additional column $v_\delta(z)$.

\begin{proposition}\label{prop_rk}
  The derivation $\delta\in V$ lies in $\AID(\g)$ if and only if for all $z\in F^n$ \[\rank(M(z)) = \rank(M_\delta(z)).\]
\end{proposition}
\begin{proof}
The derivation $\delta$ is almost-inner  if and only if for every $z\in F^n$ there exists some $x\in\g$ such that $\delta(b_z)=[b_z,x]$. Writing $x=x_1 b_1+\ldots + x_n b_n$ we have 
\[[b_z,x]=\sum\nolimits_{i,j=1}^n z_ix_j [b_i,b_j]= \sum\nolimits_{k=1}^n  (\sum\nolimits_{i,j=1}^n z_ix_j  \sigma_{i,j}^k) b_k.\]Using the definition of $v_\delta(z)$ and its components $c_k(z)$, the derivation $\delta$ lies in $\AID(\g)$ if and only if for all $z_1,\ldots,z_n\in F$, there  exist $x_1,\ldots,x_n\in F$ such that for each $k\in\{1,\ldots,n\}$
\[c_k(z) = (\sum\nolimits_{i,j=1}^n z_ix_j  \sigma_{i,j}^k)=\sum\nolimits_{j=1}^n (\sum\nolimits_{i=1}^n z_i\sigma_{i,j}^k) x_j. \]The latter holds if and only if
\begin{align}\label{eq_Mz} &M(z)\cdot (x_1,\ldots,x_n)^\intercal = v_\delta(z)
\end{align}
where $M(z)$ and $v_\delta(z)$ are as defined prior to the proposition. We have that \eqref{eq_Mz} has a solution if and only if $v_\delta(z)$ lies in the column space of $M(z)$. The claim follows.
\end{proof}

Note that $m_{k,j}(z)$ and $c_i(z)$ are (linear) polynomials in $z_1,\ldots,z_n$. Now let $z=(z_1,\ldots,z_n)$ be the vector of indeterminates of $F[z_1,\ldots,z_n]$ and consider the corresponding matrices $M(z)$ and $M_\delta(z)$.
The rank of a matrix can be defined as the largest integer $r$ such that there exists a nonzero $r\times r$ minor, that is, a nonzero determinant of an $r\times r$ submatrix. For a fixed $r$, denote by $K_{r}(z)$ and $K_{\delta,r}(z)$ the set of all $r\times r$ minors of $M(z)$ and of $M_\delta(z)$, respectively; so the elements in $K_r(z)$ and $K_{\delta,r}(z)$ are polynomials in $F[z_1,\ldots,z_n]$. Let $\mathcal{I}_r(z)$ and $\mathcal{I}_{\delta,r}(z)$ be the ideals generated by $K_r(z)$ and $K_{\delta,r}(z)$, respectively. Recall that the radical of an ideal $I$ in a ring $R$ is the ideal $\sqrt{I}=\{r\in R : r^n \in I\text{ for some $n\in\mathbb{N}$}\}$. 

\begin{proposition}
  Let $\delta\in V$ be a derivation. If $\delta$ is not an almost-inner derivation, then there exists some $r$ and $w\in \mathcal{I}_{\delta,r}(z)$ with  $w\notin \sqrt{\mathcal{I}_r(z)}$. Conversely, if the field is algebraically closed and there exist some $r$ and $w\in\mathcal{I}_{\delta,r}(z)$ such that $w\notin \sqrt{\mathcal{I}_r(z)}$, then $\delta$ is not an almost-inner derivation.
\end{proposition}
\begin{proof}
  If $\delta$ is not an almost-inner derivation, then there exists $\tilde z\in F^n$ such that $r=\rank(M_\delta(\tilde z))$ is larger than $\rank(M(\tilde z))$, that is, there is some $r\times r$ minor $w(\tilde z)$ of $M_\delta(\tilde z)$ that does not vanish, but all $r\times r$ minors of $M(\tilde z)$ are $0$. Thus, $\tilde z$ is a common root of all elements in $\mathcal{I}_r(z)$, but not of all elements in $\mathcal{I}_{\delta,r}(z)$. In particular, $w(z)$ lies in  $\mathcal{I}_{\delta,r}(z)$, but not in $\sqrt{\mathcal{I}_r(z)}$.

  Conversely, suppose that there exists some $r$ and $w\in\mathcal{I}_{\delta,r}(z)$ such that $w\notin \sqrt{\mathcal{I}_r(z)}$. Over an algebraically closed field, Hilbert's Nullstellensatz \cite[Theorem 4.1.2]{cox} says that $\sqrt{\mathcal{I}_r(z)}$ is exactly the set of all polynomials that vanish on all the common roots of the elements in $\mathcal{I}_r(z)$. This means that there is a common root $\tilde z$ for all the elements in $\mathcal{I}_r(z)$, but $\tilde z$ is not a root of $w(z)$. This implies that $\rank(M_{\delta,r}(\tilde z))$ is greater than $\rank(M_r(\tilde z))$, and therefore $\delta$ is not an almost-inner derivation by Proposition \ref{prop_rk}.
\end{proof}
  
We note that deciding  membership in the radical can be achieved without computing the radical, see \cite[Proposition 4.2.8]{cox}. 

\begin{corollary}\label{cor} Let $\delta\in V$ be a derivation.
  \begin{ithm}
  \item Over an algebraically closed field, $\delta\in\AID(\g)$ if and only if  $\mathcal{I}_{\delta,r}(z)\subseteq \sqrt{\mathcal{I}_r(z)}$ for every $r>1$.
  \item If $\mathcal{I}_{\delta,r}(z)\subseteq\sqrt{\mathcal{I}_r(z)}$ for each $r>1$, then $\delta\in\AID(\g)$ (even for non-algebraically closed fields).
  \end{ithm}
\end{corollary}

\begin{remark} If the latter condition in part a) of the corollary does not hold, then there exists some $r\times r$ minor $w(z)\in\mathcal{I}_{\delta,r}(z)$ that does not lie in $\sqrt{\mathcal{I}_r(z)}$; in particular, there exists a point $\tilde z$ such that $w(\tilde z)$ is not zero, but $\tilde z$ is a common root of all the elements in $\mathcal{I}_r(z)$. One can attempt to find $\tilde z$ by working in the polynomial ring $F[z_1,\ldots,z_n,y]$ and finding a common zero of the polynomials in the ideal generated by $K_r(z)$ and $w(z)y-1$: such a common zero annihilates every generator in $K_r(z)$, but due to $w(\tilde z)y=1$, it cannot be a zero of $w(z)$. Once such an element $\tilde z\in F^n$ is found, one can reduce $V$ to $V\cap D_{b_{\tilde z}}$. Note that the latter intersection is smaller than $V$ since $\delta\in V$, but $\delta\notin D_{b_{\tilde z}}$.  Eventually, one can iterate this method to reduce $V$ to a smaller subspace such that one can verify that each generator is an almost-inner derivation. However, it is still a computational challenge to find these suitable points $\tilde z$.  In particular, for fields that are not algebraically closed, it is not even known whether the problem of finding such points is decidable, c.f.\ Hilbert's 10th Problem over $\Q$, see \cite[Section 2.6.4]{PoonenBook}. 
\end{remark}

\section{Computational examples}\label{sec_ex}
\noindent Our methods can quickly deal with small-dimensional Lie algebras. For example,
we went through the classification of $8$-dimensional filiform Lie algebras given in \cite{ango}. For the cases where the Lie algebra depends on a parameter
we set it equal to 1. Our methods readily compute the quotient
$\AID(\g)/\Inn(\g)$, and it has dimension $0$, $1$, or $2$ in all cases. When the dimension of the Lie algebra
increases we can still quickly get a good idea of $\AID(\g)/\Inn(\g)$ by
computing the intersection of many spaces $D_{z_0}$ for arbitrarily chosen $z_0$. However, proving
that a given derivation is almost-inner by the minors-method described above quickly
becomes cumbersome due to the large number of minors that has to be considered.

We now  illustrate our method with some computations in the algebra system Magma~\cite{magma}.

\begin{example} We consider the complex Lie algebra $\g=\g_{6,23}$ of \cite[p.\ 112]{phd}. It is a $6$-dimensional Lie algebra with basis $\{b_1,\ldots,b_6\}$ whose non-zero commutators are $[b_1,b_2]=b_3$, $[b_1,b_3]=b_5$, $[b_1,b_4]=b_6$, and $[b_2,b_4]=b_5$. The complement space to $\Inn(\g)$ in $\Der(\g)$ has dimension $10$, and the intersection of a few spaces $D_z$ quickly finds a $2$-dimensional subspace $U\leq \Der(\g)$ such that $\AID(\g)\leq \Inn(\g)\oplus U$. We choose a basis $\{\delta_1,\delta_2\}$ of $U$ and let $\delta=d_1\delta_1+d_2\delta_2$, such that  Equation \eqref{eq_Mz}  reads
\[\left(\begin{matrix} -z_2 & z_1 & 0 & 0 \\ -z_3 & -z_4 & z_1 & z_2 \\ -z_4 &0 & 0 &z_1\end{matrix}\right)\cdot (x_1,x_2,x_3,x_4)^\intercal = (-d_1 z_1,-d_2z_2,0)^\intercal.\]
  To show that $\delta_1\in\AID(\g)$, we consider $d_1=1$ and $d_2=0$, and the matrices
  \[M(z) = \left(\begin{matrix} -z_2 & z_1 & 0 & 0 \\ -z_3 & -z_4 & z_1 & z_2 \\ -z_4 &0 & 0 &z_1\end{matrix}\right)\quad\text{and}\quad
    M_\delta(z)\left(\begin{matrix} -z_2 & z_1 & 0 & 0 & -z_1\\ -z_3 & -z_4 & z_1 & z_2&0 \\ -z_4 &0 & 0 &z_1&0\end{matrix}\right).\]
      The following Magma code establishes that $\mathcal{I}_{\delta,3}(z)\subseteq \sqrt{\mathcal{I}_3(z)}$ and $\mathcal{I}_{\delta,2}(z)\subseteq \sqrt{\mathcal{I}_2(z)}$, and now Corollary \ref{cor}  proves that $\delta_1$ is in $\AID(\g)$. The same computation with $M_\delta(z)$ adjusted to $\delta=\delta_2$ proves that $\delta_2\in \AID(\g)$. Thus, in this case, $\AID(\g)=\Inn(\g)\oplus U$ for some $2$-dimensional $U$.

{\footnotesize
\begin{verbatim}
    Mv  := Matrix([[-z2,z1,0,0,0],[-z3,-z4,z1,z2,-z2],[-z4,0,0,z1,0]]);
    M   := Matrix([[-z2,z1,0,0],[-z3,-z4,z1,z2],[-z4,0,0,z1]]);
    I3  := Radical(ideal<P|Minors(M,3)>);
    I3v := ideal<P|Minors(Mv,3)>;
    I2  := Radical(ideal<P|Minors(M,2)>);
    I2v := ideal<P|Minors(Mv,2)>;
    forall(i){ i : i in GroebnerBasis(I2v) | i in I2};
    //true
    forall(i){ i : i in GroebnerBasis(I3v) | i in I3};
    //true 
\end{verbatim}
}
\end{example}

\begin{example}
Now let $\g$ be the complex $5$-dimensional Lie algebra of \cite[Lemma 8.2.11]{phd}, whose non-vanishing brackets are $[b_1, b_4] = b_1$, $[b_1, b_5] = -b_2$, $[b_2, b_4] = b_2$, $[b_2, b_5] = b_1$, $[b_4, b_5] = b_3$. A complement $U$ to $\Inn(\g)$ in $\Der(\g)$ has dimension $2$, and we choose a basis $\{\delta_1,\delta_2\}$. For $\delta=\delta_1$, we obtain the following calculations: 

{\footnotesize
\begin{verbatim}
    C<i> := CyclotomicField(4);
    P<z1,z2,z3,z4,z5,x1,x2,x3,x4,x5,y,d1,d2>:= PolynomialRing(C,13);
    m   := Matrix([[-z4,-z5,0,z1,z2],[z5,-z4,0,z2,-z1],[0,0,0,-z5,z4]]);
    mv  := Matrix([[-z4,-z5,0,z1,z2,-z1],[z5,-z4,0,z2,-z1,-z2],
                   [0,0,0,-z5,z4,0]]);
    I3  := Radical(ideal<P|Minors(m,3)>);
    I3v := ideal<P|Minors(mv,3)>;
    forall(i){ i : i in GroebnerBasis(I3v) | i in I3};
    //false; next, we try to find a suitable \tilde z
    gb  := GroebnerBasis(I3);
    exists(min){min : min in Minors(mv,3) | not min in I3};
    // true
    Append(~gb, min*y - 1);
    gb  := GroebnerBasis(tmp);
    // [ z1^2*z5*y + z2^2*z5*y - 1,   z4^2 + z5^2 ]
    > [Evaluate(f,[1,1,0,-i,1,0,0,0,0,0,1/2,0,0]) : f in gb];
    // [ 0, 0 ]
\end{verbatim}}
At the end of this computation we have found $\tilde z=(1,1,0,-i,1)$ and $y=1/2$, where $i=\imath$ is a primitive $4$-th root of unity, such that $D_z$ with $z=b_1+b_2-\imath b_4+b_5$ satisfies $\dim(U\cap D_z)=1$. One can now iterate this process and eventually establish that $\AID(\g)=\Inn(\g)$. The quoted code also shows that over the subfield $\mathbb{R}$ it is not possible to find such an element $\tilde z$: the equation $z_4^2+z_5^2=0$ forces $z_4=z_5=0$, but then $z_1^2 z_5 y + z_2^2 z_5 y - 1=-1$ is never $0$.
\end{example}

\begin{example}
In \cite[Propositions 3.5 and 3.8]{BM} two examples of Lie algebras $\g$ over
fields of positive characteristic were given such that the quotient $\Der(\g)/\Inn(\g)$ is simple and non-solvable. The first example is $\mathfrak{psl}(3,F)$, where $F$ is a field of characteristic 3. The second example is the ideal $J$ generated by the short root vectors of the simple Lie algebra of type $F_4$ over  a field of characteristic 2. In both cases, by computing the intersection of a
small number of spaces $D_{z_0}$, our methods quickly show that $\AID(\g)=\Inn(\g)$.
\end{example}

\section{A non-abelian $\Sha(\g)$}\label{secsha}
\noindent We now provide an affirmative answer to \cite[Question 4.1(i)]{boris} in the positive characteristic case (\!\cite[Section 4.1.2]{boris}), namely, we construct a Lie algebra $\g$ such that $\Sha(\g)=\AID(\g)/\Inn(\g)$ is non-abelian. The Lie algebra $\g$ is defined by a finite $p$-group constructed in \cite[Theorem p.\ 67]{Sah}. We briefly describe the construction of that group, and then comment on the construction of $\g$ and computation of $\AID(\g)$.

Let $p$ be a prime, let $F=\mathbb{F}_{q}$ be the field with $q=p^3$ elements, and let $R=F^3$ be the $3$-dimensional $F$-vector space. Denote by $\{1,\pi,\pi^2\}$ an $F$-basis of $R$ and define a left multiplication on $R$ by $\pi f=f^q \pi$ for $f\in F$, and $\pi^i\pi^j = \pi^{i+j}$ if $i+j\leq 2$ and $\pi^i\pi^j=0$ otherwise. The multiplicative group $U_1(R)=1+R\pi$  acts via left multiplication on the additive group $R$, giving rise to the split extension \[G_p = (R,+)\rtimes U_1(R).\] If $\{1,\alpha,\alpha^2\}$ is a $\Z_p$-basis of $F$, then $U_1(R)$ is generated by  $\{1+\alpha^i\pi^j\mid i=0,1,2;\; j=1,2\}$, the centre $Z$ of $U_1(R)$ is generated by $\{1+\alpha^i\pi^2\mid i=0,1,2\}$, and $Z\cong U_1(R)/Z \cong (\Z/p\Z)^3$. Thus, $G_p$ can be described as an extension $C_p^9\rtimes (C_p^3.C_p^3)$ of order $p^{15}$; here $C_n$ denotes a cyclic group of order $n$.

For $p=2,3$ we have constructed $G_p$ in GAP \cite{gap} as a group given by a polycyclic presentation. For $p=2$, the group $U_1(R)$  can be reconstructed in GAP as SmallGroup(64,82); the group $G_2$ has order $2^{15}=32768$ and can be reconstructed in GAP as PcGroupCode(c,32768), where
{\footnotesize
  \begin{align*}
    c=&7967110418574553081670398114259915186817035422811746306474238743720883172561132845\\
    &8148063967477625450818367910830831150163051940361653140207335697269643407173799006\\
     &1156240486161235975.
\end{align*}}For $p=3$, the groups $U_1(R)$ and $G_3$ are isomorphic to the groups constructed in GAP as\linebreak SmallGroup(729, 122) and PcGroupCode(c,14348907), where
{\footnotesize
  \begin{align*}
    c=&23942575175932849938787447903034850559076995269050972392619102284965697037971664120\\
    &90204677927634837434203516265451902627457581534280475495170411505144605912923993727\\
    &1123090348191913725150811094324857778021113477919143155502672960892416353913323520.
\end{align*}}
There are several standard ways to attach a Lie algebra to a finite $p$-group.
Here we use the $p$-central series $G_p=G_p^1\geq G_p^2\geq  G_p^3\geq G_p^4=1$. The quotients $G_i/G_{i+1}$ are elementary abelian $p$-groups, and 
hence can be viewed as vector spaces over the field $\mathbb{F}_p$ with $p$ elements. The Lie algebra
$\g_p$ is the direct sum of these spaces and the Lie bracket on $\g_p$ is
induced by the commutator in $G_p$. We refer to \cite[\S 1.4]{graaf} for a
precise account (in this reference the Jennings series is used, but for
the $p$-central series it works in the same way).

For $\g_2$ and $\g_3$ we computed the intersection of a large number of
spaces $D_{z_0}$ for arbitrarily chosen $z_0\in \g_p$. In both cases, this quickly yields a space of
dimension $21$ of possible almost-inner derivations that are not inner.
By running over {\em all} elements $z_0\in \g_p$, we then proved that
this space indeed consists of almost-inner derivations. We note that when extending the field and considering the Lie algebras $\mathbb{F}_{8}\otimes_{\mathbb{F}_2} \g_2$ and $\mathbb{F}_{27}\otimes_{\mathbb{F}_3} \g_3$,
it is quickly shown by computing the intersection of a number of
spaces $D_{z_0}$ that the algebra of almost-inner derivations is equal to
the algebra of inner derivations. This is reflected by the fact that for the
ideals generated by minors we have many $k$ such that 
$\mathcal{I}_{\delta,k}(z)$ is not contained in $\sqrt{\mathcal{I}_k(z)}$.
However, these ideals turn out to be too complicated to be analysed 
further in detail. 

We provide some more details for $\g_3$: this algebra has basis $\{v_1,\ldots,v_{15}\}$ with multiplication table given in Figure \ref{fig_g3}. The algebra of derivations of $\g_3$ has dimension $45$, and the ideal of inner
derivations $\Inn(\g_3)$ has dimension $12$. In the subalgebra $\AID(\g_3)$ 
of almost-inner derivations satisfies $\AID(\g_3)=\Inn(\g_3)\oplus U$ where $U$ is a subalgebra of dimension $21$. In Figure \ref{fig_2} we explicitly define two non-commuting elements $d_1$ and $d_2$ of $U$.

\begin{figure}\small
\addtolength{\jot}{-0.8ex}
\begin{align*}
[v_1, v_2] &= 2v_7, & [v_3, v_4] &= 2v_{11}, \\
[v_1, v_3] &= v_7+2v_8, & [v_3, v_5] &= v_{11}+2v_{12}, \\
[v_1, v_4] &= 2v_{10}, &[v_3, v_6] &= 2v_{10}+v_{12}, \\ 
[v_1, v_5] &= v_{11}, & [v_3, v_{10}] &= 2v_{13}+v_{14}, \\
[v_1, v_6] &= v_{12}, & [v_3, v_{11}] &= v_{13}+2v_{14}+2v_{15}, \\
[v_1, v_{10}] &= 2v_{13}, & [v_3, v_{12}] &= v_{13}+v_{14}+2v_{15}, \\
[v_1, v_{11}] &= v_{14}+v_{15}, &   [v_4, v_7] &= v_{13}, \\
[v_1, v_{12}] &= 2v_{15}, & [v_4, v_8] &= 2v_{15}, \\
[v_2, v_3] &= 2v_8+v_9, & [v_4, v_9] &= v_{14}+v_{15}, \\
[v_2, v_4] &= 2v_{12}, & [v_5, v_7] &= v_{14}+v_{15}, \\
[v_2, v_5] &= 2v_{10}+v_{12}, & [v_5, v_8] &= 2v_{13}+v_{15}, \\
[v_2, v_6] &= v_{11}, & [v_5, v_9] &= 2v_{14}+v_{15}, \\
[v_2, v_{10}] &= v_{13}+2v_{15}, & [v_6, v_7] &= 2v_{15}, \\
[v_2, v_{11}] &= v_{13}+2v_{14}+v_{15}, & [v_6, v_8] &= 2v_{14}+2v_{15}, \\
[v_2, v_{12}] &= v_{14}+2v_{15}, & [v_6, v_9] &= 2v_{13}+v_{15}. 
\end{align*}
\caption{Structure constants for  $\g_3$; if $[v_i,v_j]$ is not listed, then $[v_i,v_j]=0$.}\label{fig_g3}
\end{figure}

\begin{figure}\small
\addtolength{\jot}{-0.8ex}
\begin{align*}\small
d_1( v_2 )&=v_7,  & d_2( v_2 )&=  v_{10},\\
d_1( v_3 )&=2v_7 + v_8, & d_2( v_7 )&=  v_{13},\\
d_1( v_4 )&= 2v_{11} + 2v_{12}, & d_2( v_8 )&=  v_{13},\\
d_1( v_5 )&= 2v_{10}, & d_2( v_9 )&=  2v_{13} + 2v_{14},\\
d_1( v_6 )&= v_{10}, &\\
d_1( v_{10} )&= v_{14}, &\\ 
d_1( v_{11} )&=  v_{13}, &\\
d_1( v_{12} )&=  2v_{13}.
\end{align*}
\caption{Definitions of two non-commuting elements of $U$; if $d_i(v_k)$ is not listed, then  $d_i(v_k)=0$. }\label{fig_2}
\end{figure}


\end{document}